\documentclass[letterpaper, 10 pt, conference]{ieeeconf}  

\IEEEoverridecommandlockouts                              
\usepackage{amsmath} 
\usepackage{comment}
\excludecomment{dontshow}
\usepackage[utf8]{inputenc} 
\usepackage[T1]{fontenc}    
\usepackage{hyperref}       
\usepackage{url}            
\usepackage{booktabs}       
\usepackage{amsfonts}       
\usepackage{nicefrac}       
\usepackage{microtype}      
\usepackage{float}          
\usepackage{amssymb}
\usepackage{algorithm,algpseudocode}
\usepackage{setspace}
\usepackage{lmodern}
\usepackage{paralist}  
\usepackage{multirow}
\usepackage{tablefootnote}

\usepackage{graphicx}
\usepackage{algorithm}
\usepackage{algpseudocode}

\usepackage[english]{babel}
\usepackage{color}

\newtheorem{theorem}{Theorem}[section]

\newtheorem{proposition}[theorem]{Proposition}

\newtheorem{definition}[theorem]{Definition}
\usepackage{chngcntr}
\usepackage{apptools}
\AtAppendix{\counterwithin{theorem}{subsection}}

\algnewcommand\algorithmicinit{\textbf{Initialization:}}
\algnewcommand\Initialization{\item[\algorithmicinit]}
\algnewcommand\algorithmicinitend{\textbf{end initialization}}
\algnewcommand\InitializationEnd{\item[\algorithmicinitend]}
\algnewcommand\algorithmicpiv{\textbf{SCLP-pivot}}
\algnewcommand\SCLPPivot{\item[\algorithmicpiv]}

\algnewcommand\algorithmicswitch{\textbf{switch}}
\algnewcommand\algorithmiccase{\textbf{case}}
\algnewcommand\algorithmicassert{\texttt{assert}}
\algnewcommand\Assert[1]{\State \algorithmicassert(#1)}%
\algdef{SE}[SWITCH]{Switch}{EndSwitch}[1]{\algorithmicswitch\ #1\ \algorithmicdo}{\algorithmicend\ \algorithmicswitch}%
\algdef{SE}[CASE]{Case}{EndCase}[1]{\algorithmiccase\ #1}{\algorithmicend\ \algorithmiccase}%
\algtext*{EndSwitch}%
\algtext*{EndCase}%

\newcommand{\Tt}{^{{\mbox{\tiny \bf \sf T}}}}
\def\dx{\dot{x}}
\def\dq{\dot{q}}
\newcommand{\J}{\mathcal{J}}
\newcommand{\Ss}{\mathcal{S}}
\newcommand{\N}{\mathcal{N}}

\newcommand{\setR}{\mathbb{R}}
\newcommand{\K}{{\cal{K}}}

\newcommand{\bJ}{{\mathbf{J}}}
\newcommand{\bK}{{\mathbf{K}}}

\algdef{SE}[SWITCH]{Switch}{EndSwitch}[1]{\algorithmicswitch\ #1\ \algorithmicdo}{\algorithmicend\ \algorithmicswitch}%
\algdef{SE}[CASE]{Case}{EndCase}[1]{\algorithmiccase\ #1}{\algorithmicend\ \algorithmiccase}%
\algtext*{EndSwitch}%
\algtext*{EndCase}%
\let\oldReturn\Return
\renewcommand{\Return}{\State\oldReturn}
\title{\LARGE \bf
SCLP-Simplex Algorithm for Robust Fluid Processing Networks
}

\author{Evgeny Shindin, Roi Ben Gigi and Odellia Boni 
\thanks{E. Shindin is with IBM Research - Israel, Mount Carmel, Haifa, 3498825, Israel
        {\tt\small evgensh@il.ibm.com}}%
\thanks{R. Ben Gigi with IBM Research - Israel, Mount Carmel, Haifa, 3498825, Israel
        {\tt\small roi.ben.gigi@ibm.com}}%
\thanks{O. Boni with IBM Research - Israel, Mount Carmel, Haifa, 3498825, Israel
        {\tt\small odelliab@il.ibm.com}}%
}

\begin{document}

\maketitle
\thispagestyle{empty}
\pagestyle{empty}

\begin{abstract}
 Fluid models provide a tractable approach to approximate multiclass processing networks. This tractability is a due to the fact that optimal control for such models is a solution of a Separated Continuous Linear Programming (SCLP) problem. Recently developed revised SCLP-simplex algorithm allows to exactly solve very large instances of SCLPs in a reasonable time. Furthermore, to deal with the inherent stochasticity in arrival and service rates in processing networks, robust optimization approach is applied to SCLP models. However, a robust counterpart of SCLP problem has two important drawbacks limiting its tractability. First, the robust counterpart of SCLP problem is a huge SCLP problem itself, that can be in several orders of magnitude bigger then the nominal SCLP problem. Second, robust counterpart of SCLP is a degenerate optimization problem, that is not suitable for revised SCLP-simplex algorithm. In this paper we develop theoretical results and a corresponding algorithm that allows to preserve dimensions of nominal SCLP problem and avoid degeneracy issues during solution of its robust counterpart. 
\end{abstract}

\section{Introduction} 
\label{sec:introduction}
In a multi-class processing network, entities of different classes arrive at a service and undergo sequential processing stages by various servers. Each processing stage for each class is associated with a distinct operation, and each server has the ability to execute multiple operations. Such networks are widely used to model and analyze various real-world systems, including telecommunication networks, computer systems, manufacturing processes, transportation systems, and service-oriented systems like call-centers and healthcare facilities. Clearly, these models demand effective control over entity admissions, routing, sequencing, and operational scheduling to optimize network performance. Finding an optimal control policy requires solving stochastic dynamic programming models. However, for real-life networks such models are extremely large and computationally intractable. Alternatively, fluid models provide a viable approximation approach, offering asymptotically optimal control (see \cite{Nazarathy2009}). 

Finding optimal control for fluid models involves solving a Separated Continuous Linear Programming (SCLP) problem, which poses challenges due to its infinite-dimensional nature. Several methods has been suggested to solve SCLP problems, including time discretization  approaches \cite{Pullan1993,Luo1998,Fleischer2005,Pullan2002}, polynomial approximation methods \cite{Bampou2012} and simplex-type algorithms \cite{Weiss2008, Shindin2018, Shindin2021}. In \cite{Shindin2021} authors showed that for SCLP that originated from fluid approximation of  processing networks, a revised SCLP-Simplex algorithm outperforms time-discretization approaches both in speed and in solution quality.

In real-life networks arrival and/or processing rates often are not known exactly or may change over time. One of the common ways to deal with uncertainty in optimization problems parameters is to apply robust optimization methodology (e.g. \cite{BenTal2009}). Robust optimization assumes that the uncertain parameters reside in a region known as \emph{uncertainty set} and aims to formulate another deterministic optimization problem known as \emph{robust counterpart} (RC) such that  each solution of RC is a feasible solution of the original problem for  all possible combinations of parameters in the uncertainty set. In other words, it satisfies the original problem for the worst-case realization of the uncertain parameters. Robust counterparts of fluid processing networks were presented in \cite{Bertsimas2014} and further extended in \cite{Ship2022}, where different problem formulations and different tractable uncertainty sets were considered. 

Unfortunately, revised SCLP-simplex algorithm \cite{Shindin2021} is not suitable for solving RC of uncertain SCLP problem, by the following reasons:
\begin{compactitem}
\item RC contains intrinsic degeneracy. Both primal and dual formulations becomes degenerate, while SCLP-simplex does not support degenerate problems. 
\item RC contains many additional primal state variables. This greatly affects solution time, because these variables participate in all computationally intensive steps of SCLP-simplex algorithm. 
\end{compactitem}


The main contribution of this paper is an efficient method for solving uncertain SCLP formulations of processing networks where processing rates of different servers belong to a budgeted uncertainty set. Our approach includes a reduction method that transforms portions of the budgeted uncertainty set into box uncertainty sets, eliminating the need for additional variables or constraints. Furthermore, we introduce a cutting planes algorithm tailored to uncertain SCLPs, seamlessly integrated into the Revised SCLP-Simplex algorithm with minimal modifications of the later.  
We demonstrate both contributions for the case of one-sided budgeted uncertainty set of service effort model (\ref{eqn:ASCLP}). However,  similar approach can be applied to processing rates model and/or other polyhedral uncertainty sets. 


\section{Background}
\label{sec:background}
\subsection{Robust fluid processing networks}
\label{sec:robust_sclp}
Consider a fluid processing network with $I$ servers, $K$ buffers, and $J$ flows. Each flow can be processed by a dedicated server, and each server can handle several flows. Additionally, each flow empties a specific buffer, and there could be several flows emptying the same buffer. We denote $s(j) = i$ if flow $j$ is processed by server $i$, and $f(j) = k$ if flow $j$ empties buffer $k$. After processing, the fluid either moves to another buffer or exits the system. Let $p_{k,j}$ represent the proportion of flow $j$ that moves to buffer $k$ after processing. The following quantities are related to buffer $k$:
\begin{compactitem}
\item 
$x_k(t)$ amount of fluid at time $t$,
\item
$\alpha_k = x_k(0)$ initial amount of fluid, 
 \item
 $a_k$ constant exogenous input rate,
 \item
 $g_k \ge 0$ constant holding cost per unit/time.
\end{compactitem}
Likewise, the following quantities are related to flow $j$:
\begin{compactitem}
\item 
$\eta_j(t)$ proportion of effort of server $s(j)$ dedicated to this flow at time $t$,
\item
 $\mu_j$ service rate per flow unit if server $s(j)$ works on this flow with full effort,
 \item
 $u_j(t) = \mu_j \eta_j(t)$ actual service rate per flow unit at time $t$,
 \item
 $h_j$ constant processing cost per unit/time.
\end{compactitem}
Finally, let $\mu$ be uncertain, so that $\mu:=\mu(\xi(t)), \xi(t) \in \mathcal{U}$, where $\xi(t)$ is a perturbation vector and $\mathcal{U}$ is the uncertainty set.

The goal is to find optimal proportions of server efforts $\eta(t)$ over the planned time interval $[0,T]$, subject to buffer and server capacity constraints and for all possible realizations of service rates.
The robust optimal control for this uncertain fluid processing network can be found by solving following problem:
\begin{equation}
\begin{array}{ll}
\label{eqn:ASCLP}
\displaystyle \max_{\eta(t), x(t)} & \int_0^T (\gamma\Tt + (T-t)c\Tt) (\eta(t) {\circ} \mu(\xi(t))) \,dt,   \\
 \mbox{s.t.} &  \int_0^t G\, (\eta(s) {\circ} \mu(\xi(t))) \,ds  + x(t) = \alpha + a t, \\
 & \quad\; H \eta(t) \le b, \\
& \quad x(t), \eta(t) \ge 0, \quad 0\le t \le T,\, \xi(t) \in \mathcal{U},
\end{array}
\end{equation}
where, $b=1, c=g\Tt G, \gamma = -h$ and:
\[
\textstyle
G_{k,j} = \begin{cases} -p_{k,j}  & \text{ if } f(k) \ne j,\\
1 & \text{ if } f(k) = j,
\end{cases}\;
H_{i,j} = \begin{cases} 1 & \text{ if } s(j) = i,\\
0 & \text{ if } s(j) \ne i.
\end{cases} 
\]

Following \cite{Ship2022}, we consider $h=\gamma=0$, define a one-sided budgeted uncertainty set, and reformulate the RC problem. Let $\overline{\mu}$ be nominal service rate, $\tilde{\mu}$ be its maximal deviation from the nominal value, and $\Xi(t)$ be a perturbations vector such that: $\Xi_j(t) \in [0,1], \sum_{s(j)=i} \Xi_j(t) \le \Gamma_i$, where $\Gamma_i$ is an uncertainty budget associated with server $i$. Then, the maximal processing rate of flow $j$ is expressed by $\mu(\Xi(t)) = \overline{\mu} - \tilde{\mu} \Xi(t)$. 

In addition, we introduce following notations, that will be used through the rest of the paper:
\[
\overline{G}_{k,j} {=} G_{k,j} \overline{\mu}_j, \, \tilde{G}_{k,j} {=}{-} G_{k,j} \tilde{\mu}_j, \; \overline{c}_{j} {=} c_{j} \overline{\mu}_j, \, \tilde{c}_{j} {=} c_{j} \tilde{\mu}_j (\forall k,j),
\]
Then, the RC (\ref{eqn:ASCLP}) takes the form 
%
(see \cite{Ship2022}):
\begin{equation}
\arraycolsep=1.5 pt
\begin{array}{ll}
 \label{eqn:SCLP-RC1}
\displaystyle \max_{\eta, \beta,\gamma,y} & \int_0^T (T-t)\overline{c}\Tt \eta(t) - y(t)\,dt,   \\
 \mbox{s.t.} &  \sum\limits_{j=1}^J \int_0^t \overline{G}_{k,j} \eta_j(s)\,ds  {+} \sum\limits_{i=1}^I \Gamma_i \beta_{k, i}(t) {+}\\
 &  \sum\limits_{i=1}^I \sum\limits_{j: s(j)=i} \gamma_{k, i, j} {\le} {=} \alpha_k {+} a_k t, (\forall k,t)\\
  &  \beta_{k, i}(t) {+} \gamma_{k, i, j}(t) {\ge} \int_0^t \tilde{G}_{k,j} \eta_j(s)\,ds \\
  & \forall t, k, i, j: s(j) = i,\\
 & y(t) \ge \sum\limits_{i=1}^I \left(\Gamma_i \beta_{0,i}(t) + \sum\limits_{j: s(j)=i} \gamma_{0,i,j}(t)\right), \\
 & \beta_{0,i}(t) {+} \gamma_{0,i,j}(t) {\ge}\int_0^t \tilde{c}_j \eta_j(s)ds\, \forall i, s(j){=}i \\
 &  H \eta(t) \le b, \\
 & \quad \eta(t), y(t), \beta(t), \gamma(t) \ge 0, \; 0\le t \le T, 
\end{array}
\end{equation}

One can see that compared to the nominal problem (\ref{eqn:ASCLP} with no uncertainty), the RC contains additional $(K+1)\times(J+I) + 1$ primal state variables and same number of constraints (and slack variables), making it much harder to solve.

\subsection{SCLP-Simplex Algorithm}
In this section we recall structure of optimal solution of SCLP and SCLP-Simplex algorithm as described in \cite{Weiss2008, Shindin2021}.
In \cite{Weiss2008, Shindin2021} authors consider the following SCLP problem:
\begin{equation}
\begin{array}{ll}
\label{eqn:PWSCLP}
\displaystyle \max_{u(t), x(t)} & \int_0^T (\gamma\Tt + (T-t)c\Tt) u(t) \,dt,   \\
 \mbox{s.t.} &  \int_0^t Gu(s) \,ds  + Fx(t)  + x(t) = \alpha + a t, \\
 & \quad\; H u(t) \le b, \\
& \quad x(t), u(t) \ge 0, \quad 0\le t \le T, 
\end{array}
\end{equation}
where $F$ is $K\times L$ dimensional matrix that corresponds to additional state variables $x_{K+1},\dots, x_{K+L}$. One can check that the nominal version of (\ref{eqn:ASCLP}) is a subclass of (\ref{eqn:PWSCLP}). 

Denote by $u_{J+1},\ldots,u_{J+I}$ slacks of the second set of constraints of (\ref{eqn:PWSCLP}) and let
 $\bK = (1,\ldots,K+L)$  be the indexes of the primal state variables $x_k(t)$ and $\bJ=(1,\ldots,J+I)$  be the  indexes of the primal control variables $u_j(t)$.  
The symmetric dual to (\ref{eqn:PWSCLP}) is 
\begin{equation}
\begin{array}{ll}
\label{eqn.DWSCLP1}
\displaystyle \min_{p(t),q(t)} & \int_0^T (\alpha\Tt + (T-t)a\Tt) p(t) + b\Tt q(t) \,dt,    \nonumber  \\
\mbox{s.t.} &  \int_0^t  G\Tt\, p(s)\,ds + H\Tt q(t) \ge \gamma + c t, \\
 & \quad\; F\Tt p(t) \ge d, \nonumber \\
& \quad q(t), p(t)\ge 0, \quad 0\le t \le T,   \nonumber
\end{array}
\end{equation}
with dual state variables, including slacks,  $q_j(t),\,j\in \bJ$ and dual control variables $p_k(t),\,k\in\bK$.  
Note that {\em the dual problem runs in reversed time.}

Under easily checked feasibility and boundedness conditions, and under  non-degeneracy,  SCLP has a unique strongly dual solution. 
The optimal solution has piecewise constant primal and dual controls and continuous piecewise linear primal and dual state variables,   with breakpoints $0=t_0<t_1<\cdots <t_N=T$.   The solution is then fully described by the breakpoints, by the initial state values $x(0)=x^0,\,q(0)=q^N$,  and by the values of the controls and of the derivatives of the states $u_j^n=u_j(t)\,,p_k^n=p_k(T-t)$, $\dx_k^n=\dx_k(t),\,\dq_j^n=\dq_j(T-t)$ for $t_{n-1}<t<t_n,\,n=1,\ldots,N$.
The values of the primal and dual states at the breakpoints are $x_k^n=x_k(t_n),\,q_j^n=q_j(T-t_n),\,n=0,\ldots,N$.

The initial values, $x^0,q^N$, are optimal solutions of the {\em Boundary-LP}:
\begin{equation}
\label{eqn.boundary}
\begin{array}{lll}
\max \quad    [0 \; d\Tt] x^0,  &\quad&  \min  \quad   [b\Tt \; 0] q^N,   \\
\mbox{s.t.}  \quad     [I \; F] x^0 = \alpha,  &&  \mbox{s.t.}  \quad     [H\Tt  -I] q^N = \gamma,   \\
\qquad \quad  x^0 \ge 0,   &&   \quad \qquad  q^N\ge 0.  \\
\end{array}
\end{equation}
with $\K_0$, $\J_{N+1}$ the indexes of the basic variables $x_k^0, q_j^N$.
Note that Boundary-LP  does not involve $T$, so that $x^0,q^N$ are the same for all time horizons.

Values of the controls and slopes of states in the intervals are complementary slack basic solutions of the primal and dual {\em Rates-LP}$(\K,\J)$:
 \begin{equation}
\label{eqn.rates1}
\arraycolsep=2pt
\begin{array}{cc}
\begin{array}{ll}
\max &  [c\Tt\;0] u +  [0\;d\Tt] \dx    \\
\mbox{s.t.}& [G\;0] u + [I\;F] \dx = a,   \\
  &   [H\;I] u \quad \quad \quad  = b,  
\end{array}
&\quad
\begin{array}{l}
 \dx_k \in \setR \;\forall k \in \K,\\  \dx_k \in \setR^+ \;\forall k \notin \K, \\
 u_j = 0 \;\forall j \in \J,\\  u_j \in \setR^+ \;\forall j \notin \J,
 \end{array}
 \end{array}
\end{equation}
 \begin{equation}
\label{eqn.rates2}
\arraycolsep=2pt
\begin{array}{cc}
\begin{array}{cl}
\min&   [a\Tt\;0] p + [b\Tt\;0] \dq   \\
\mbox{s.t.}&  [G\Tt\;0] p + [H\Tt\;\mbox{-}I]\dq = c ,   \\
&   [F\Tt\;{\mbox{-}I}] p \quad \quad \quad   = d,
\end{array}
&\quad
\begin{array}{l}
\dq_j \in \setR \;\forall j \in \J,\\  \dq_j \in \setR^+ \;\forall j \notin \J, \\
p_k = 0 \;\forall k \in \K,\\  p_k \in \setR^+\; \forall k \notin \K,  
\end{array}
\end{array}
\end{equation}
where for interval $(t_{n-1},t_n)$ the primal basis is $B_n=\{u^n_j,\dx^n_k :  j\not\in\J_n,k\in\K_n\}$
with complementary dual basis $B^*_n =\{p^n_k,\dq^n_j : k\not\in\K_n,j\in\J_n\}$.

The bases 
have the following properties:  
\newline -- Compatibility to the boundary:  $\K_0 \subseteq \K_1$, $\J_{N+1} \subseteq \J_N$.
\newline -- Adjacency:  $B_n,B_{n+1}$ are adjacent:  in the pivot $B_n \to B_{n+1}$ a single basic variable $v^n$ leaves the basis and a single basic variable $w^n$ enters.

The breakpoints $t_1,\ldots,t_{N-1}$ are determined by the following equations for the interval lengths $\tau_n=t_n - t_{n\mbox{-}1}$:
\begin{equation}
\label{eqn.breakponts}
\arraycolsep=2pt
\begin{array}{ll} 
v^n{=}\dx_k {\implies} x_k(t_n){=} 0 {\implies}  & \sum_{m=1}^n \dx_k^m \tau_m {=} {-}x_k^0, \\
v^n{=}u_j {\implies} q_j(T{-}t_n) {=} 0 {\implies} &  \sum_{m=N}^{n+1}  \dq_j^m \tau_m {=} {-}q_j^N, \\
& \tau_1+\cdots+\tau_N {=} T. 
\end{array}
\end{equation}
The remaining values are determined by:
\begin{equation}
\label{eqn.othervalues}
\begin{array}{l} 
x_k(t_n) =  x_k^0 + \sum_{m=1}^n \dx_k^m \tau_m,  \\
q_j(T-t_n) = q_j^N + \sum_{m=N}^{n+1}  \dq_j^m \tau_m. 
\end{array}
\end{equation}
Given a sequence of adjacent bases $\mathcal{B} = \{B_n\}_{n=1}^N$ we can calculate all the controls and slopes of states, the breakpoints, and the values of the primal and dual states at all breakpoints. 
It is an optimal base sequence if: 
\begin{theorem}[\cite{Weiss2008}]
\label{thm.optimal}
If a sequence of bases $\{B_n\}_{n=1}^N$ are compatible with $\K_0,\J_{N+1}$ and are adjacent, and if all the values of the primal and dual state variables and the interval lengths determined by equations (\ref{eqn.boundary})--(\ref{eqn.othervalues})  are positive, then this  is an optimal solution of the SCLP.  \end{theorem}
The SCLP-simplex algorithm \cite{Weiss2008, Shindin2021} solves SCLP parametrically, by increasing the time horizon $\theta T$ over $0 < \theta \le 1$, with iterations needed at $0<\theta_1 <\cdots <\theta_M =1$ as follows.
 \begin{algorithm}[h!]
\caption{SCLP-Simplex Algorithm}\label{alg:sclp_simplex}
\begin{algorithmic}[1]
\Initialization{}
\State{Solve (\ref{eqn.boundary}) to obtain $x^0, q^N$}
\State{Set $\K_0 =\{k: x^0_k > 0\}, \J_{N+1} =\{j : q^N_j > 0\}$}
\State Solve Rates-LP$(\K_0, \J_{N+1})$ and obtain $B_1$ \label{alg:simplex:line_init_solveLP}
\State{Set $\ell:= 1, \theta_0 :=0, \mathcal{B}_1 := \{B_1\}$}
\InitializationEnd{}
 \Loop 
 \State Extract $\dx^n, \dq^n, n=1,\dots,N$ from $\mathcal{B}_\ell$ \label{alg:simplex:line_extract_rates}
 \State Solve (\ref{eqn.breakponts}) to obtain $\tau$  
 \State Solve (\ref{eqn.breakponts}) with RHS $[\mathbf{0}, 1]\Tt$ to obtain $\delta\tau$
 \State Compute $x^n, \delta x^n, q^n, \delta q^n$ by (\ref{eqn.othervalues})
 \State Compute $\Delta {=} \min\limits_{\delta(\cdot) {<} 0} \left\lbrace 1, \text{-}\frac{\tau_n}{\delta \tau_n},   
 \text{-}\frac{x^n_k}{\delta x^n_k},  \text{-}\frac{q^n_j}{\delta q^n_j}\right\rbrace$
 \State Set $\theta_{\ell} {:=} \theta_{\ell-1} + \Delta$
 \If{$\theta_{\ell} \ge 1$}  \State Set $\Delta{:=} 1{-}\theta_{\ell{-}1}, \tau{:=}\tau {+}\Delta \delta\tau, x^n{:=}x^n{+}\Delta \delta x^n,$  $q^n{:=}q^n{+}\Delta \delta q^n$
 \State Compute $u^n, p^n$ from $\mathcal{B}_{\ell}$
 \Return $\tau, x^n, q^n, u^n, p^n, \dx^n, \dq^n$
 \EndIf
 \State Classify collision $\mathcal{V} = \text{arg}\min\limits_{\delta(\cdot) {<} 0} \left\lbrace \text{-}\frac{\tau_n}{\delta \tau_n},   
 \text{-}\frac{x^n_k}{\delta x^n_k},  \text{-}\frac{q^n_j}{\delta q^n_j}\right\rbrace$
 \If{$\mathcal{V} = \{\tau_{n'},\dots,\tau_{n''}\}$} \State{Remove $B_{n'},\dots, B_{n''}$ from $\mathcal{B}_\ell$}
 \If {$n'{=} 0$ or  $n''{=}N$ or $\|B_{n''+1}{\setminus} B_{n'-1}\|{=}1$} \State Go to line \ref{alg:simplex:line_iter_end}
 \Else: \State $B'{:=}B_{n'-1}, B''{:=}B_{n''+1}, B'' \setminus B' = \{v',v''\}$
 \EndIf
 \Else\If{$\mathcal{V} = \{x_k^n\}$} \State $B':=B_{n}, B'':=B_{n+1}, v'=v^n, v''=\dx_k$
 \Else {$\; \mathcal{V} = \{q_j^n\}$} \State $B':=B_{n-1}, B'':=B_{n}, v'=u_j, v''=v^n$
 \EndIf
 \EndIf
 \State Set $\K^* {=} \{k: \dx_k {\in} B'\}\setminus v'', \J^* {=} \{j: u_j {\not\in} B''\}{\setminus} v'$
 \State Solve Rates-LP$(\K^*, \J^*)$ to obtain $D$ \label{alg:simplex:line_iter_solveLP}
 \If{$D$ is not adjacent to $B'$ or $B''$} \State Solve SCLP sub-problem, get $\{D\}_{m=1}^M$
 \Else: \State $\{D\}_{m=1}^M = \{D\}$
 \EndIf
 \State Update $\mathcal{B}_\ell$ insert $\{D\}_{m=1}^M$ between $B'$ and $B''$
 \State Set $\mathcal{B}_{\ell+1} := \mathcal{B}_{\ell}, \ell:=\ell+1 $ \label{alg:simplex:line_iter_end}
 \EndLoop
\end{algorithmic}
\end{algorithm}
 For collisions at 0 or $T$ some steps of the algorithm are slightly modified.

\subsection{Cutting planes algorithm} \label{sec:cutting_planes_algo}
An alternative to reformulating RC and solving it is applying the cutting planes algorithm  \cite{Gilmore1961,Gilmore1963} to the original (nominal) problem.
In this section we briefly describe this algorithm.
Consider the following problem:
\begin{equation}
\label{eqn:general_uncertainty}
\underset{\mathbf{x}\in\mathbb{R}^{n}}{\min}\; f\left(\mathbf{x}\right),\; 
\text{s.t.}\; g_i \left(\mathbf{x;\theta_i}\right) {\le} b_i, \; \theta_i {\in} \mathcal{U}_i, \; \forall i {=} 1, \dots, I
\end{equation}
where $\{\mathcal{U}_i\}$ are uncertainty sets. The cutting planes for the problem (\ref{eqn:general_uncertainty}) has following structure. 
\begin{algorithm}
\caption{Cutting Planes Algorithm}\label{alg:cutting_planes}
\begin{algorithmic}[1]

\State{Initialization of master problem with set of nominal values $\{\bar{\theta}_i\}$ as uncertain parameters, and solving it to obtain $x^*$}
 \For{each uncertain constraint i:}
        \State compute the set $\mbox{argmax}_{\theta} g_i(x^*;\theta)$
        \State for every $\tilde{\theta} \in \mbox{argmax}_{\theta} g_i(x^*;\theta) $ satisfying   $g_i(x^*;\tilde{\theta}) > b_i$, add the constraint $g_i(x;\tilde{\theta}) \leq b_i$ to master problem
\State{Solve the master problem} obtaining solution $x^*$
\EndFor
\If{during step 2 none constraint was added to master problem} \State $x^*$ is an optimal robust solution to the problem
\Else:
\State return to step 2
\EndIf
\end{algorithmic}
\end{algorithm}

For a wide class of problems and uncertainty sets following result holds (see e.g. \cite{Mutapcic2009}):
\begin{theorem}
\label{thm:gen_cutting}
   Under appropriate assumptions cutting planes algorithm \ref{alg:cutting_planes} produces a robust optimal solution for problem (\ref{eqn:general_uncertainty}).
\end{theorem}

\section{Reduction of uncertainty set}
\label{sec:reduction}
Analyzing the structure of the uncertain SCLP problem (\ref{eqn:ASCLP}) and the structure of uncertainty set one can reduce the size of the RC. In this section we present an alternative form of the RC of the server-effort model for the one-sided budgeted uncertainty set, discuss reduction of this RC, and present a reduction algorithm. 

\begin{theorem}
\label{thm:new-rc}
   The RC of the uncertain SCLP problem (\ref{eqn:ASCLP}) can be expressed by:
\begin{equation}
\arraycolsep=1.5 pt
\begin{array}{ll}
 \label{eqn:SCLP-RC2}
\displaystyle \max_{\eta,\beta,\gamma} & \int_0^T (T\text{-}t)\mkern-6mu\left(\overline{c}\Tt \eta(t) {-}\mkern-6mu\sum\limits_{i=1}^I \mkern-3mu(\Gamma_i \beta_{0,i}(t) {+} \mkern-18mu \sum\limits_{j: s(j)=i}\mkern-18mu\gamma_{0,i,j}(t)\mkern-3mu)\mkern-6mu\right)\mkern-6mudt,   \\
 \mbox{s.t.} &  \int_0^t \mkern-6mu\left(\sum\limits_{j=1}^J  \overline{G}_{k,j} \eta_j(s) {+} \sum\limits_{i=1}^I \Big( \Gamma_i \beta_{k, i}(s) {+} \right. \\
 &   \left. \sum\limits_{j: s(j)=i}\mkern-18mu \gamma_{k, i, j}(s)\Big)\mkern-6mu\right) \,ds {+} x_k(t) {=} \alpha_k {+} a_k t, (\forall k,t)\\
  &  \beta_{k, i}(t) {+} \gamma_{k, i, j}(t) {\ge} \tilde{G}_{k,j} \eta_j(t) \,\forall t, k, i, s(j) {=} i,\\
 & \beta_{0,i}(t) {+} \gamma_{0,i,j}(t) {\ge} \tilde{c}_j \eta_j(t), \forall t,i, s(j){=}i, \\
 &  H \eta(t) \le b, \\
 & \; x(t), \eta(t), \beta(t), \gamma(t) \ge 0, \; 0\le t \le T.
\end{array}
\end{equation}
\end{theorem}
\begin{proof}
The proof of the theorem given in Appendix.       
\end{proof} 

Unlike (\ref{eqn:SCLP-RC1}), in formulation (\ref{eqn:SCLP-RC2}) we add primal control variables instead of primal states. 

Unfortunately, the reformulation (\ref{eqn:SCLP-RC2}) does not reduce the problem size, so we still need to apply further reduction. Recall uncertain constraint of (\ref{eqn:ASCLP}) with one-sided budgeted uncertainty set as it described in Section \ref{sec:robust_sclp}. To ensure that constraint $k$ holds under all possible realizations of $\Xi(t)$ one should consider the following:
\begin{equation}
\label{eqn:RSCLP_constr_k}
\begin{array}{l}
    \sum\limits_{j=1}^J \int\limits_0^t \overline{G}_{k,j} \eta_j(s) ds + z_k(t) \le \alpha_k + a_k t \\
    \begin{array}{ll}
        z_k(t) =  &  \max\limits_{\Xi(t)} \sum\limits_{j=1}^J \int\limits_0^t \Xi_{j}(s) \tilde{G}_{k,j} \eta_j(s) ds \\
        \text{s.t.} & \sum\limits_{j: s(j) = i} \Xi_{j}(t) \le \Gamma_i\, \forall i, \quad 0 \le \Xi(t) \le 1 
    \end{array}
\end{array} 
\end{equation}
Recall that elements of $G$ are either proportion of the endogenous input flows $-p_{k,j}$, or $1$ for output flows from buffer $k$. Thus, for the flow $j$ that does not transfer the fluid into the buffer $k$ we have $\tilde{G}_{k,j} = 0$, and for $j: f(j)=k$ we have $\tilde{G}_{k,j} = -\tilde{\mu}_j$. Moreover, one can see that optimization problem in (\ref{eqn:RSCLP_constr_k}) could be divided into $I$ sub-problems that will take following form:
\begin{equation}
\label{eqn:RSCLP_sub}
    \begin{array}{ll}
        z_{k,i}(t) =  & \max\limits_{\Xi(t)} \sum\limits_{j \in \N_{i,k}} \int\limits_0^t \Xi_{j}(s) \tilde{G}_{k,j} \eta_j(s) ds \\
        \text{s.t.} & \sum\limits_{j \in \N_{i,k}} \Xi_{j}(t) \le \Gamma_i, \quad 0 \le \Xi(t) \le 1 
    \end{array}
\end{equation}
where $\N_{i,k} = \{j: \{s(j) {=} i\} \cap \{\tilde{G}_{k,j} {>} 0\}\}$. One can see, that if $\|\N_{i,k}\| \le \Gamma_i$ problem (\ref{eqn:RSCLP_sub}) has trivial solution $\Xi(t) = 1$ with objective values $z_{k,i}(t) = \sum_{j \in \N_{i,k}} \int_0^t \tilde{G}_{k,j} \eta_j(s) ds$.
This leads to the following algorithm:
\begin{algorithm}
\caption{Reduction Algorithm For Robust SCLP}
\label{alg:reduction}
\begin{algorithmic}[1]
\For{$(k=1; k{++}; k \le K)$}
    \State Set $N_i = 0 \forall i, \mathcal{R}_k =\emptyset$
    \For{$(j=1; j{++}; j \le J)$}
        \If{$G_{k,j} > 0$}
           \State Set $N_{s(j)} = N_{s(j)} + 1$ 
        \EndIf
    \EndFor
    \For{$(j=1; j{++}; j \le J)$}
        \If{$N_{s(j)} \le \Gamma_i$}
            \State $\overline{G}^*_{k,j} = \overline{G}_{k,j} + \tilde{G}_{k,j}$
        \Else:
            \State $\overline{G}^*_{k,j} =  \overline{G}_{k,j},\; \mathcal{R}_k = \mathcal{R}_k  \cup j$
        \EndIf                    
    \EndFor
\EndFor
\end{algorithmic}
\end{algorithm}

Applying algorithm  \ref{alg:reduction} to problem (\ref{eqn:ASCLP}) we obtain an uncertain SCLP problem, where the first set of constraints takes the following form:
\begin{equation}
\textstyle
    {\int_0^t}\mkern-6mu\left({\sum\limits_{j=1}^J}  \overline{G}^*_{k,j} \eta_j(s) {+} {\sum\limits_{j \in \mathcal{R}_k}} \tilde{G}_{k,j} \Xi_j(s) \eta_j(s)\mkern-6mu\right) ds {\le} \alpha_k {+} a_k t
\end{equation}
This problem has the reduced set of the uncertain parameters and could be further transformed into the RC, or solved by a cutting planes algorithm that will be presented in Section \ref{sec:cutting}. Note that the objective function of (\ref{eqn:ASCLP}) could be treated similarly.

\section{Cutting planes for uncertain SCLP}
\label{sec:cutting}

In this section, we present a cutting planes algorithm for uncertain SCLP, derived from the application of cutting planes algorithm \ref{alg:cutting_planes} to problem (\ref{eqn:ASCLP}). Recall that the RC of uncertain SCLP (\ref{eqn:SCLP-RC2}) is an SCLP problem, and SCLP has an optimal solution where unknown controls $\eta(t)$ are piecewise constant functions of $t$. Thus, for each time interval $n=1,\dots,N$ where $\eta(t) = \eta^n$ is constant, the solution of problem (\ref{eqn:RSCLP_sub}) does not depend on $t$,  and hence problem (\ref{eqn:RSCLP_sub}) can be considered as:
\begin{equation}
\label{eqn:RSCLPN_sub}
    \begin{array}{ll}
        z_{k,i} =  &  \max\limits_{\Xi} \sum\limits_{j \in \N_{i,k}} \sum\limits_{n=1}^N \Xi_{j,n} \tilde{G}_{k,j} \eta^n_j \\
        \text{s.t.} &  \sum\limits_{j \in \N_{i,k}} \Xi_{j,n} \le \Gamma_i, \quad 0 \le \Xi \le 1 
    \end{array}
\end{equation}
It is easy to check that this problem could be further separated into the set of smaller sub-problems for each $n$:
\begin{equation}
\label{eqn:RSCLPN_sub_sub}
    \begin{array}{ll}
         z_{k,i,n} =  & \max\limits_{\Xi} \sum\limits_{j \in \N_{i,k}} \Xi_{j,n} \tilde{G}_{k,j} \eta^n_j \\
        \text{s.t.} &  \sum\limits_{j \in \N_{i,k}} \Xi_{j,n} \le \Gamma_i, \quad 0 \le \Xi \le 1 
    \end{array}
\end{equation}
Note that according to SCLP-simplex algorithm \ref{alg:sclp_simplex}, $\eta^n$ is the solution of Rates-LP (\ref{eqn.rates1}), and hence we can plug in (\ref{eqn:RSCLPN_sub_sub}) back into Rates-LP. Applying a similar approach to the objective, we get the following optimization problem:
 \begin{equation}
\label{eqn:un_rates1}
\arraycolsep=1.5pt
\begin{array}{cc}
\begin{array}{ll}
\max &  [(\overline{c} -\tilde{c}\Xi)\Tt\;0] \eta    \\
\mbox{s.t.}& (\overline{G} + \tilde{G}\Xi) \eta + I \dx = a,   \\
  &   [H\;I] \eta \quad \quad \quad  = b,  
\end{array}
&\;
\begin{array}{l}
 \dx_k \in \setR \;\forall k \in \K,\\  \dx_k \in \setR^+ \;\forall k \notin \K, \\
 \eta_j = 0 \;\forall j \in \J,\\  \eta_j \in \setR^+ \;\forall j \notin \J,
 \end{array}
 \end{array}
\end{equation}
where $ 0 \le \Xi \le 1,\, \sum_{j: s(j)=i} \Xi_{j} \le \Gamma_i \forall i$.

It follows that optimal solution of Robust SCLP (\ref{eqn:SCLP-RC2}) could be obtained from the sequence of robust optimal solutions of uncertain Rates-LP (\ref{eqn:un_rates1}). The latter could be solved by the cutting planes algorithm and hence we discuss the cutting planes algorithm for (\ref{eqn:un_rates1}).

One can observe that the second set of constraints in (\ref{eqn:un_rates1}) does not depend on uncertainty. Moreover, for $k \in \K$, constraint $k$ holds for any realization of uncertainty, so we only need to solve (\ref{eqn:RSCLPN_sub_sub}) for $k \notin \K$.
Furthermore, solving problem (\ref{eqn:RSCLPN_sub_sub}) for specific $k, i, n$ we obtain:
\begin{equation}
\label{eqn:RLP_constr_k_sol}
\begin{cases}
\Xi_{k,j} {=} 1 & \mkern-6mu j {\in} \Ss^k_{i, \eta^*}(\lfloor \Gamma_i \rfloor), \\
\Xi_{k,j} {=} (\Gamma_i {-} \lfloor \Gamma_i \rfloor)\Psi_{k,i} & \mkern-6mu j {\in} \Ss^k_{i, \eta^*}(\lceil \Gamma_i \rceil) {\setminus} \Ss^k_{i, \eta^*}(\lfloor \Gamma_i \rfloor), \\
0  & \mkern-6mu \text{else}, 
\end{cases}
\end{equation}
where $\Ss^k_{i, \eta}(L)$ is any set of indices of the greatest $L$ elements in the vector $\tilde{G}_k \circ \eta$ satisfying $s(j) {=} i$ and $\tilde{G}_{k,j} \eta_j {>} 0$, and
$\Psi_{k,i}= \| \Ss^k_{i, \eta^*}(\lceil \Gamma_i \rceil) {\setminus} \Ss^k_{i, \eta^*}(\lfloor \Gamma_i \rfloor)\|$.
It should be noted that the optimal solution (\ref{eqn:RLP_constr_k_sol}) may not be unique and depends on the choice of $\Ss^k_{i, \eta}(L)$. It is evident from (\ref{eqn:RLP_constr_k_sol}) that $\Xi_{k,j}$ takes a finite set of values, enabling the enumeration of all possible combinations of these values as $1,\dots,M$. Consequently, for any $\eta^* \in \mathbb{R}^J$ and for each $k$, there exists an optimal solution of (\ref{eqn:RSCLPN_sub_sub}) in the form of (\ref{eqn:RLP_constr_k_sol}) denoted by $\hat{\Xi}^m_{k}$.
We define $\hat{G}_{k} = \tilde{G}_k \circ \hat{\Xi}^m_k$ as the worst-case realization of parameters for constraint $k$, and we define:
\begin{equation}
\label{eqn:worst_case_constr_k}
\textstyle
    \sum_{j} \left(\overline{G}_{k,j} + \hat{G}_{k,j} \right) \eta^*_j \le a_k
\end{equation}
is a worst case realization of constraint $k$ for the some $\eta^*$ if $\hat{G}_{k} = \tilde{G}_k \circ \hat{\Xi}^m_k$, where $\hat{\Xi}^m_k$ is optimal solution of (\ref{eqn:RSCLPN_sub_sub})  for this $\eta^*$. Similarly, we define $\Ss^0_{i,\eta^*}(L)$ and $\hat{\Xi}^m_{0,j}$ for the objective function, denoting by $\hat{c} = \tilde{c} \circ \hat{\Xi}^m_0$ the worst-case realization of the objective coefficients.

These leads to the following cutting planes algorithm for the Rates-LP$(\K^*, \J^*)$.
\begin{algorithm}
\caption{Cutting Planes Algorithm For Rates-LP}
\label{alg:cutting_planes_LP}
\begin{algorithmic}[1]
\Loop $\quad \ell=1,\dots$
\State Solve Rates-LP/LP$^*(\K^*, \J^*)$ and get $\eta^*$
 \For{$k\not\in \K^*$}
        \State For all $i$ compute $\Ss^{k,\ell}_{i, \eta^*}(\lfloor \Gamma_i \rfloor), \Ss^{k,\ell}_{i, \eta^*}(\lceil \Gamma_i \rceil)$
        \State Get $\hat{\Xi}^\ell_k$ from (\ref{eqn:RLP_constr_k_sol}) and set $\hat{G}^\ell_{k} {:=} \tilde{G}_k \circ \hat{\Xi}^\ell_k$
        \If{$\sum_{j} \left(\overline{G}_{k,j} + \hat{G}^\ell_{k,j} \right) \eta^*_j > a_k$}
            \State Add constraint (\ref{eqn:worst_case_constr_k}) to Rates-LP$(\K^*, \J^*)$
        \EndIf
\EndFor
\If{none constraint was added} 
    \State Get $\Ss^{k,\ell}_{i, \eta^*}$ for $\lfloor \Gamma_i \rfloor, \lceil \Gamma_i \rceil$, $\hat{G}^\ell_{k} {:=} \tilde{G}_k \circ \hat{\Xi}^\ell_k$ $\forall k {\in} \K^*$
    \State Get $\Ss^0_{i, \eta^*}$ for $\lfloor \Gamma_i \rfloor, \lceil \Gamma_i \rceil$, $\hat{c} {:=} \tilde{c} \circ \hat{\Xi}_0$, $z {=} (\overline{c} {-} \hat{c})\Tt \eta$  
    \Return $\eta^*, \dx {=} \min_\ell a_k {-} \sum_{j} \left(\overline{G}_{k,j} {+} \hat{G}^\ell_{k,j} \right) \eta^*_j$\label{alg:cutting_planes_LP:line_return}
\EndIf
\EndLoop
\end{algorithmic}
\end{algorithm}

\begin{theorem}
\label{thm:prim_rates}
    Let $\dx, \eta$  be the solution obtained from the Algorithm \ref{alg:cutting_planes_LP}, then the same $\dx, \eta$ is the robust optimal solution of the uncertain Rates-LP$(\K^*,\J^*)$ (\ref{eqn:un_rates1}).
\end{theorem}
\begin{proof}
The proof of the theorem given in Appendix.       
\end{proof} 

Recall that the SCLP-simplex algorithm \ref{alg:sclp_simplex} provides the optimal solution for an SCLP problem with a different set of parameters at each step. Thus, instead of solving the nominal SCLP problem up to the target time horizon $T_{\text{goal}}$ and then applying the cutting planes algorithm to the nominal solution of the SCLP, we apply the cutting planes algorithm \ref{alg:cutting_planes} at each iteration of SCLP-simplex, obtaining a solution of robust SCLP along the entire parametric line. This leads to the following.
\begin{definition}
\label{def:modification}
To solve the uncertain SCLP problem \ref{eqn:ASCLP}, the SCLP-simplex algorithm \ref{alg:sclp_simplex} requires the following modifications:

(i) After lines \ref{alg:simplex:line_init_solveLP} and \ref{alg:simplex:line_iter_solveLP}, the cutting planes algorithm \ref{alg:cutting_planes_LP} is applied to the nominal solution of the Rates-LP.

(ii) At line \ref{alg:simplex:line_extract_rates}, $\dx$ is the result of Algorithm \ref{alg:cutting_planes_LP} (line \ref{alg:cutting_planes_LP:line_return}), and $\dq$ is extracted from the dual solution obtained at the final stage of Algorithm \ref{alg:cutting_planes_LP}.
\end{definition}

\begin{theorem}
\label{thm:sclp_alg}
The SCLP-simplex algorithm \ref{alg:sclp_simplex}, with modifications \ref{def:modification}, provides a robust optimal solution to the uncertain SCLP. That is, $\eta(t), x(t)$ obtained from the modified algorithm are optimal for the RC problem (\ref{eqn:SCLP-RC2}).
\end{theorem}
\begin{proof}
The proof of the theorem given in Appendix.       
\end{proof}

\section{Results and discussion}
\label{sec:discussion}
\subsection{Reduction algorithm}
\subsubsection{Setup of the experiment}
The performance of the reduction algorithm \ref{alg:reduction} was evaluated using a set of randomly generated uncertain SCLP problems. To simplify the testing, we considered a model without routing, resulting in a single outflow for each buffer. The total number of servers denoted by  $I = 10 \iota$, where $\iota =1,\dots,10$ \, and the number of buffers was chosen in different proportions as $K = 2 m I$, where $m=1,\dots,5$. The total number of input buffers for each flow was randomly drawn from the integer uniform distribution $n \sim \text{Unif}(1, \theta K)$, where $\theta$ is a given parameter and then the set of size $n$ of input buffers is randomly generated. Flows served by specific servers were chosen randomly, and the uncertainty budget for each server was set proportionally to the number of served flows, defined as $\Gamma_i = \kappa \sum_{j:s(j)=i} 1$. For each combination of parameters, we generated 10 random problems and calculated the average number of additional variables required to construct the RC after applying the reduction algorithm. To simplify computations, uncertainty in the objective function was not taken into account.

The number of additional variables before reduction depends solely on the problem dimensions and is given by $K \times (K+I)$.  Thus, we calculated the relative reduction as $R = 100\% - ($number of additional variables after reduction$)/($total number of additional variables$)$. 
\subsubsection{Results} We found that the relative reduction depends on the number of input buffers per flow and the uncertainty budget, while the problem size and number of buffers do not significantly affect the reduction.Figure \ref{fig:graph1} presents average data on relative reduction for different sizes of budgets and different numbers of input buffers for each flow.
\begin{figure}[H]
\label{fig:graph1}
\centering
\includegraphics[scale=0.5]{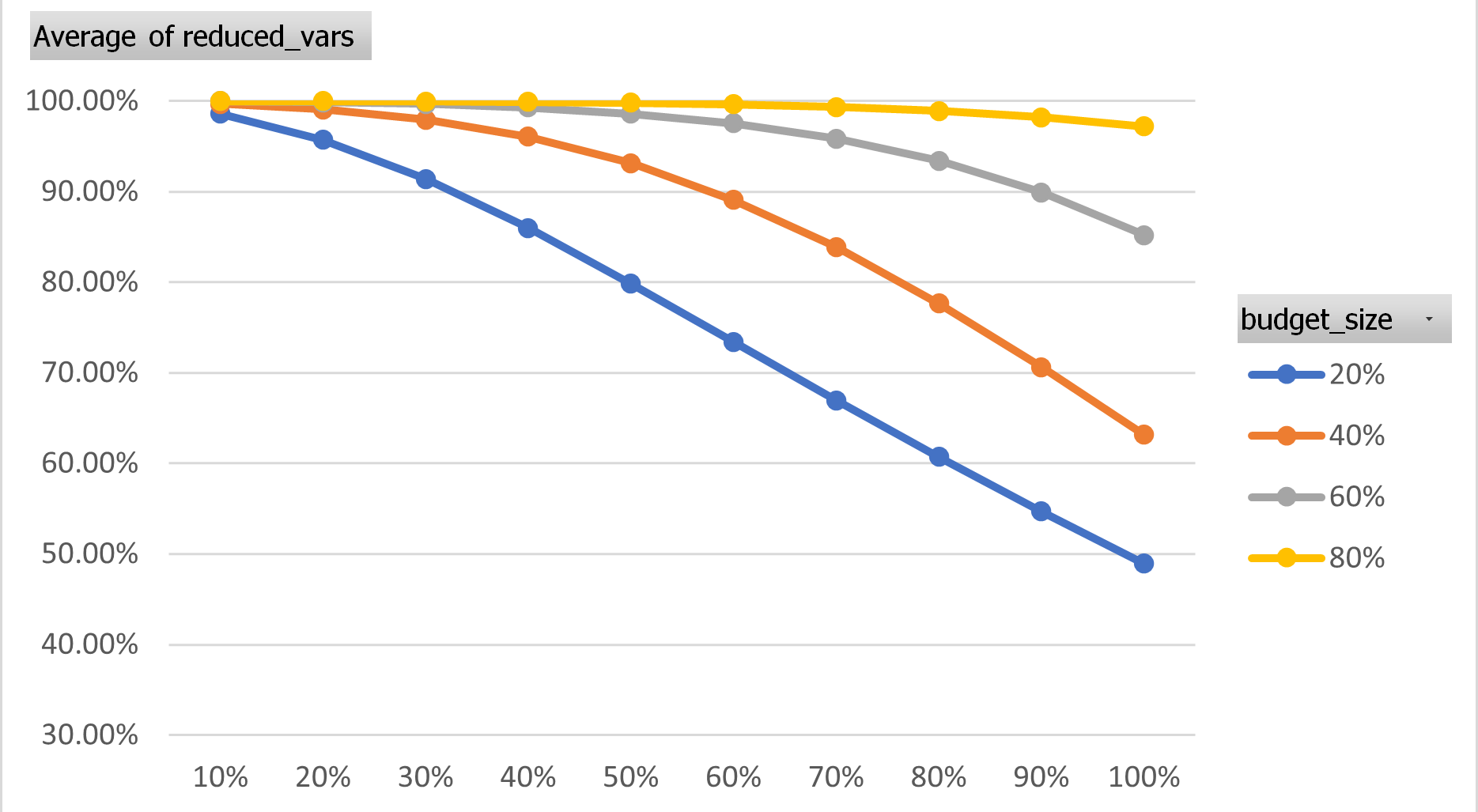}
\caption{Relative number of reduced variables}
\end{figure}
Here the horizontal axis represent the values of $\theta$ and different lines corresponds to different values of $\kappa$.
It can be observed that the number of reduced variables varies from almost  $100\%$ to $50\%$, increasing with the size of the budget and decreasing with the number of input buffers per flow.

\subsection{Cutting planes for SCLP}
The implementation of the cutting planes method for SCLP is currently pending, and hence, we lack empirical results. Nevertheless, we can evaluate the efficiency of the proposed approach by comparing it to methods for obtaining a robust optimal solution for uncertain SCLP that starts from the construction of RC (\ref{eqn:SCLP-RC1} or (\ref{eqn:SCLP-RC2}.

Firstly, it is possible to discretize the time and then solve the resulting LP. However, this method appears unpromising. Numerical study of deterministic SCLP, as conducted in \cite{Shindin2021}, indicates that the solution obtained from the discretized LP could be significantly far from the optimum (with relative error reaching up to $100\%$ for a $10\times$ discretization) or may require considerable computation time (up to $1000\times$ more than the Revised SCLP-Simplex algorithm for a $1000\times$ discretization). Moreover, the RC introduces a substantial number of additional constraints and variables to the nominal problem, leading to a considerable increase in problem dimensions. For instance, a relatively small model with $10$ servers and $100$ buffers could yield an RC with over $10000$ additional variables and constraints. Time discretization worsens this issue, leading to the construction of an LP with over $10$ million variables and constraints, which is essential for achieving accurate solution.

Secondly, it is possible to perturb RC and then solve by Revised SCLP-Simplex. The drawbacks of this approach have already been discussed in Section \ref{sec:introduction}. Moreover, perturbation introduces many small time intervals, breaking the structure of the optimal solution, so it will be required to restore solution of the original problem. Finally, the Revised SCLP-Simplex algorithm has been tested only up to the scale of $10000$ variables and constraints and may encounter numerical instabilities when problem dimensions exceed these limits. Hence, even a relatively modest uncertain model with $10$ servers and $100$ buffers could become too large for this method.

On the other hand, the cutting planes algorithm for SCLP offers several advantages:
\begin{compactitem}
    \item It does not introduce additional primal and/or dual state variables to the SCLP, ensuring that computationally intensive steps (such as those in \ref{eqn.breakponts} and \ref{eqn.othervalues}) are performed on SCLP of original dimensions.
    \item It affects only the Rates-LP, and in general, for uncertain LPs with polyhedral uncertainty sets, cutting planes are faster than solving the corresponding RC, as demonstrated in \cite{Fischetti2012} and \cite{Bertsimas2016} through numerical studies.
    \item It requires to add worst-case realizations only to the subset of constraints of Rates-LP$(\K^*, \J^*)$, so that for a large $\K^*$ number of additional computations will be relatively small.
    \item Since the bases of Rates-LP are adjacent, it is possible to initiate the solution of Rates-LP with new sign restrictions from the robust optimal solution of the adjacent Rates-LP, thereby minimizing the number of iterations of the cutting planes algorithm (\ref{alg:cutting_planes_LP}). 
\end{compactitem}
\section{Conclusion and Feature work}
In this paper we presented an efficient approach to solve uncertain SCLP problems. We plan to implement the cutting planes algorithm and perform numerical study, similar to the one considered in \cite{Shindin2021}. Additionally, we aim to explore the adjustable robust methodology for uncertain SCLP, developing theories and algorithms similar to those established for LP problems.

\bibliography{ERC_bib1} 
\bibliographystyle{ieeetr}      
 \appendix
 \subsection{Proofs}
\label{appendix:proofs}
\begin{proof}\emph{[Theorem \ref{thm:new-rc}]} 
In order to build RC of (\ref{eqn:ASCLP}), recall that the first constraint set could be formulated as (\ref{eqn:RSCLP_constr_k}), where the optimization problem could be further decoupled to a set of sub-problems (\ref{eqn:RSCLP_sub}).
The symmetric dual of (\ref{eqn:RSCLP_sub}) takes the following form (see e.g. \cite{Shapiro2001}):
\begin{equation}
\label{eqn:RSCLP_sub*}
\arraycolsep=1.5pt
    \begin{array}{ll}
        z^*_{k,i}(t) {=}  &  \min\limits_{\beta(t), \gamma(t)} \int\limits_0^t \Gamma_i \beta_{k,i}(s) {+} \sum\limits_{j \in \N_{i,k}} \gamma_{k,i,j}(s)  ds, \\
        \text{s.t.} & \beta_{k,i}(t) {+} \gamma_{k,i,j}(t) {\ge} \tilde{G}_{k,j} \eta_j(t) \, (\forall j {\in} \N_{i,k}),\\
        & \beta(t), \gamma(t) \ge 0,
    \end{array}
\end{equation}
where $\N_{i,k} = \{j: \{s(j) = i\} \cap \{\tilde{G}_{k,j} > 0\}\}$. One can see that replacement of $\N_{i,k}$  to $\N_{i,k}' = \{j: s(j) = i\}$ in (\ref{eqn:RSCLP_sub}) and (\ref{eqn:RSCLP_sub*}) does not affect constraints or objective values. Thus, substituting (\ref{eqn:RSCLP_sub*})
for all $k,i$ into (\ref{eqn:ASCLP}) we obtain first and second set of constraints of (\ref{eqn:SCLP-RC2}).

Similarly, the objective functional of (\ref{eqn:ASCLP}) could be represented by:
\begin{equation}
\label{eqn:RSCLP_obj}
\begin{array}{l}
    \int\limits_0^T (T-t) \left(\overline{c}\Tt \eta(t) - z_0(t) \right) dt  \\
    \begin{array}{ll}
        z_0(t) =  & \max\limits_{\Xi(t)} \int_0^T \sum\limits_{j=1}^J  \Xi_{j}(t) \tilde{c}_{j} \eta_j(t) dt \\
        \text{s.t.} & \sum\limits_{j: s(j) = i} \Xi_{j}(t) \le \Gamma_i\, \forall i, \quad 0 \le \Xi(t) \le 1 
    \end{array}
\end{array} 
\end{equation}
Following the same arguments one can decompose this and then formulate symmetric dual problems, resulting in the objective and in the third set of constraints of (\ref{eqn:SCLP-RC2}).
\end{proof}
\begin{proof}\emph{[Theorem \ref{thm:prim_rates}]}
We are solving linear programming problem with a polyhedral uncertainty set by the cutting planes method. Thus, by Theorem \ref{thm:gen_cutting}, Algorithm \ref{alg:cutting_planes_LP} produces a robust optimal solution of the uncertain Rates-LP$(\K^*. \J^*)$.
\end{proof}

To proof Theorem \ref{thm:sclp_alg} we need to establish several results, related to the Rates-LP.
\begin{proposition}
  RC of uncertain Rates-LP$(\K^*, \J^*)$ (\ref{eqn:un_rates1}) has the following form:
  \begin{equation}
\arraycolsep=1.5 pt
\begin{array}{ll}
 \label{eqn:RatesLP-RC2}
\displaystyle \max_{\eta,\beta,\gamma} & \overline{c}\Tt \eta  - \sum\limits_{i=1}^I \left(\Gamma_i \beta_{0,i} + \sum\limits_{j: s(j)=i}\gamma_{0,i,j}\right)   \\
 \mbox{s.t.} &  \sum\limits_{j=1}^J \overline{G}_{k,j} \eta_j  {+} \mkern-6mu \sum\limits_{i=1}^I\mkern-3mu \left(  \mkern-3mu \Gamma_i \beta_{k, i} {+}\mkern-18mu \sum\limits_{j: s(j)=i} \mkern-18mu \gamma_{k, i, j} \mkern-6mu \right)\mkern-6mu  {+} \dot{x}_k {=} a_k\, \forall k, \\
  &  \beta_{k, i} {+} \gamma_{k, i, j} {-} \tilde{G}_{k,j} \eta_j {-} v_{k,i,j} {=} 0,\, \forall k, i, j{:} s(j) {=} i,\\
   & \beta_{0, i} {+} \gamma_{0,i,j} {-} \tilde{c}_j  \eta_j  - r_{i,j} = 0,\, \forall i, j{:} s(j){=}i, \\
 &  [H I]\eta = b, \\
& \dx_k \ge 0 \text{ if } k \not\in \K^*,\; \beta, \gamma, v, r \ge 0, \\
& \eta_j = 0 \text{ if } j \in \J^*,\; \eta_j \ge 0 \text{ if } j \not\in \J^*.
\end{array}
\end{equation}
\end{proposition}
\begin{proof}
Recall that constraints of Rates-LP$(\K^*, \J^*)$ (\ref{eqn:un_rates1}) holds for all possible realization of the uncertainty if and only if $\Xi$ is an optimal solution of the problem (\ref{eqn:RSCLPN_sub_sub}). The dual problem for (\ref{eqn:RSCLPN_sub_sub}) is:
\begin{equation}
\label{eqn:RSCLPN_sub_sub*}
\arraycolsep=1.5pt
    \begin{array}{ll}
        z^*_{k,i}(t) {=}  & \min\limits_{\beta, \gamma} \Gamma_i \beta_{k,i} + \sum\limits_{j \in \N_{i,k}} \gamma_{k,i,j}, \\
        \text{s.t.} &  \beta_{k,i} + \gamma_{k,i,j} \ge \tilde{G}_{k,j} \eta_j \, (\forall j \in \N_{i,k}),\\
        & \beta, \gamma \ge 0,
    \end{array}
\end{equation}

Similarly, the objective function of  (\ref{eqn:un_rates1}) could be represented by $\overline{c}\Tt \eta - z_0$, where $z_0$ is an optimal solution of the following problem:
\begin{equation}
\label{eqn:RSCLPN_sub_sub_obj}
    \begin{array}{ll}
        \displaystyle z_0 =  & \max\limits_{\Xi} \sum\limits_{j=1}^J  \Xi_{j} \tilde{c}_{j} \eta_j \\
        \text{s.t.} & \sum\limits_{j: s(j) = i} \Xi_{j} \le \Gamma_i\, \forall i, \quad 0 \le \Xi \le 1 
    \end{array}
\end{equation}
Once can see, that formulating the dual of (\ref{eqn:RSCLPN_sub_sub_obj}) and substituting it together with (\ref{eqn:RSCLPN_sub_sub*}) back to the Rates-LP$(\K^*, \J^*)$ (\ref{eqn:un_rates1})  and then introducing slack variables we obtain RC that given by (\ref{eqn:RatesLP-RC2}).
\end{proof}
\begin{proposition}
\label{prop:prim_optimal}
Optimal solution of (\ref{eqn:RatesLP-RC2}) could be obtained from the cutting planes algorithm \ref{alg:cutting_planes_LP} by the same $\eta^*, \dx^*$, and by:
\[
\beta^*_{k,i} =\mkern-24mu \min_{j\in \Ss^{k,\overline{\ell}}_{j, \eta^*}(\lceil\Gamma_i\rceil)}\mkern-12mu \{0, \hat{G}^{\overline{\ell}}_{k,j} \eta^*_j \}, \; \gamma^*_{k,i,j} {=} \max_{s(j)=i}\{0, \tilde{G}_{k,j}\eta^*_j - \beta^*_{k,i}\},
\]
where $\overline{\ell} = \arg\max_\ell \sum_j \hat{G}^{\ell}_{k,j} \eta^*_j$.
\end{proposition}
\begin{proof}
By Theorem \ref{thm:prim_rates} the cutting planes algorithm \ref{alg:cutting_planes_LP} provides an optimal solution for the uncertain Rates-LP$(\K^*, \J^*)$ (\ref{eqn:un_rates1}). Hence, we need to check the feasibility of the obtained solution and equality of the objective values.
By construction we have $\beta^*, \gamma^* \ge 0$ and
\[
\beta^*_{k,i} + \gamma^*_{k,i,j} - \tilde{G}_{k,j} \eta^*_j \ge 0
\]
Furthermore, one can see that:
\[
\textstyle \sum\limits_{j \in \Ss^k_{j, \eta^*}(\lceil\Gamma_i\rceil)} \mkern-24mu \hat{G}^{\overline{\ell}}_{k,j} \eta^*_j =  \mkern-24mu\sum\limits_{j \in \Ss^k_{j, \eta^*}(\lceil\Gamma_i\rceil)}  \mkern-24mu \hat{\Xi}_{k,j} \tilde{G}_{k,j} \eta^*_j = \Gamma_i \beta^*_{k,i} {+}\mkern-12mu \sum\limits_{j:s(j){=}i} \gamma^*_{k,i,j}
\]
where $\hat{\Xi}$ is solution of (\ref{eqn:RSCLPN_sub_sub}) given by (\ref{eqn:RLP_constr_k_sol}). Hence:
\begin{align*}
\textstyle \sum\limits_{j=1}^J \overline{G}_{k,j} \eta^{*}_j  + \sum\limits_{i=1}^I \left(   \Gamma_i \beta^{*}_{k, i} {+} \sum\limits_{j: s(j)=i}  \gamma^{*}_{k, i, j}  \right) + \dot{x}^{*}_k {=} a_k
\end{align*}
By the similar arguments for the objective values we have:
\begin{align}
\label{eqn:prim_objectives}
\textstyle (\overline{c} - \hat{c})\Tt \eta^* = \overline{c}\Tt \eta^* - \sum\limits_{i=1}^I \left(   \Gamma_i \beta^{*}_{0, i} {+} \sum\limits_{j: s(j)=i}  \gamma^{*}_{0, i, j}  \right) 
\end{align}

\end{proof}
The dual problem of (\ref{eqn:RatesLP-RC2}) is given by:
\begin{equation}
\arraycolsep=1.5 pt
\begin{array}{ll}
 \label{eqn:RatesLP-RC2*}
\displaystyle \min_{p, \dot{q}} & a\Tt p  + [b\Tt\; 0] \dq   \\
 \mbox{s.t.} & \mkern-6mu \sum\limits_{k=1}^K\mkern-6mu \left(\overline{G}_{k,j} p_k  {+}  \tilde{G}_{k,j}  \dot{\delta}_{k,j} \right)\mkern-6mu  {+} \tilde{c}_j  \dot{\delta}_{0,j} {+}\mkern-18mu \sum\limits_{i=J+1}^{J+I} \mkern-15mu H_{i,j} \dq_i {-} \dq_{j} {=} \overline{c}_j\\
  &  p_k - \dot{\delta}_{k,j} - y_{k,j} = 0,\, \forall k, j,\\
  &  \dot{\delta}_{0,j} + y_{0,j} = 1, \forall j,\, \\
   & \Gamma_i p_k - \sum\limits_{j:s(j)=i} \dot{\delta}_{k,j} - \omega_{k,i} = 0,\, \forall k, i  \\
 &  \sum\limits_{j:s(j)=i} \mkern-6mu \dot{\delta}_{0,j} + \omega_{0,j} = \Gamma_i,\, \forall i,\\ 
 & p_k \ge 0 \text{ if } k \not\in \K^*, \; p_k = 0 \text{ if } k \in \K^*,\\
& \dq_j \ge 0 \text{ if } j \not\in \J^*,\; \dot{\delta}, y, \omega \ge 0.
\end{array}
\end{equation}
while the dual problem for the problem obtained in the final iteration of the cutting planes algorithm \ref{alg:cutting_planes_LP} is:
\begin{equation}
\arraycolsep=1.5 pt
\begin{array}{ll}
\label{eqn:dual_cutting}
 \min\limits_{p,\dq} & \sum\limits_{k=1}^K \sum\limits_{\ell {\in} L(k)}   a_k p_{k, \ell}  + \sum\limits_{i=1}^{I} b_{i} \dq_i    \\
 \mbox{s.t.} & \mkern-6mu  \sum\limits_{\mkern-3mu k=1}^K\mkern-3mu\sum\limits_{\ell {\in} L(k)} \mkern-12mu (\overline{G}_{k,j} {+} \hat{G}^{\ell}_{k,j}) p_{k,\ell} {+}\mkern-6mu \sum\limits_{i=1}^{I}\mkern-6mu  H_{i,j} \dq_{i} {-} \dq_j {=} \overline{c}_j {-} \hat{c}_j,\\
 & \mkern-16mu p_{k,\ell} {\ge} 0 \text{ if } k {\not\in} \K^*, \; p_{k,\ell} {=} 0 \text{ if } k {\in} \K^*, \dq_j {\ge} 0 \text{ if } j {\not\in} \J^*
\end{array}
\end{equation}
where by $L(k)$ we denote the set of iterations where constraint $k$ has been added to the Rates-LP$(\K^*, \J^*)$, including iteration $0$, where the nominal problem has been solved.

Note, that following relation holds between optimal solutions of (\ref{eqn:dual_cutting}) and (\ref{eqn:RatesLP-RC2*}).
\begin{theorem}
    Let $p_{k,\ell}^*, \dq^*$ be the optimal solution of (\ref{eqn:dual_cutting}), then the optimal solution of (\ref{eqn:RatesLP-RC2*}) given by: $\dq' = \dq^*,$ 
\begin{equation*}
\begin{aligned}
&  \textstyle p_k' {=}\mkern-6mu \sum\limits_{\ell \in L(k)}\mkern-6mu p^*_{k,\ell},\; \dot{\delta}_{0,j}' {=} 1 \text{ if } \hat{c}_j {=} \tilde{c}_j,\; \dot{\delta}_{0,j}' {=} 0 \text{ if } \hat{c}_j {=} 0,\\
& \textstyle \dot{\delta}_{k,j}' {=}  \sum\limits_{\ell:\{\hat{G}^\ell_{k,j} {=} \tilde{G}_{k,j}\}} p_{k,\ell}^* {+} \sum\limits_{\substack{{\ell:\{\hat{G}^\ell_{k,j} {=} (\Gamma_{s(j)}}\\{{-}\lfloor\Gamma_{s(j)}  \rfloor) \tilde{G}_{k,j}\}}}}  (\Gamma_{s(j)} {-} \lfloor\Gamma_{s(j)}  \rfloor) p^*_{k,\ell},\\
& \textstyle \dot{\delta}_{0,j}' {=} (\Gamma_{s(j)} {-} \lfloor\Gamma_{s(j)}  \rfloor) \text{ if } \hat{c}_j {=} (\Gamma_{s(j)} {-} \lfloor\Gamma_{s(j)}  \rfloor) \tilde{c}_j.
\end{aligned}
\end{equation*}
\end{theorem}
\begin{proof}
 First, we show that $p',\dot{\delta}', \dq'$ is a feasible solution of (\ref{eqn:RatesLP-RC2*}). It is straightforward to check that the first set of constraints of (\ref{eqn:RatesLP-RC2*}) holds. Furthermore, by construction we have:
 \[
 \textstyle
 p'_k - \dot{\delta}_{k,j}' {=} \mkern-12mu \sum\limits_{\ell \in L(k)} \mkern-12mu p_{k,\ell}^* {-} \mkern-24mu \sum\limits_{\ell:\hat{G}^\ell_{k,j} {=} \tilde{G}_{k,j}} \mkern-18mu p_{k,\ell}^* {-}\mkern-24mu \sum\limits_{\substack{{\ell:\{\hat{G}^\ell_{k,j} {=} (\Gamma_{s(j)}}\\{{-}\lfloor\Gamma_{s(j)}  \rfloor) \tilde{G}_{k,j}\}}}}\mkern-36mu (\Gamma_{s(j)} {-} \lfloor\Gamma_{s(j)}  \rfloor) p^*_{k,\ell} {\ge} 0
 \]
and $\dot{\delta}_{0,j}'\le 1$.
Moreover, recall that $\hat{G}^{\ell}_{k,j} = \tilde{G}_{k,j} \hat{\Xi}^\ell_{k,j}$, and note, that by (\ref{eqn:RLP_constr_k_sol}) for each $\ell, i$ we have maximum $\lfloor\Gamma_{i} \rfloor$ of $\Xi^{\ell}_{k,j}=1$ and maximum one $\Xi^{\ell}_{k,j}=\Gamma_{i} - \lfloor\Gamma_{i} \rfloor$, and hence: 
\[
\textstyle
\Gamma_i p'_k {-}\mkern-12mu \sum\limits_{j:s(j)=i}\mkern-12mu \dot{\delta}'_{k,j} {\ge}\mkern-12mu \sum\limits_{\ell\in L(k)}\mkern-12mu (\Gamma_i p^*_{k,\ell} {-}\mkern-18mu \sum\limits_{\Ss^{k,\ell}_{i,\eta^*}(\lfloor\Gamma_i\rfloor)}\mkern-18mu p^*_{k,\ell} {-} (\Gamma_{i} {-} \lfloor\Gamma_{i} \rfloor) p_{k,\ell}^*) {\ge} 0
\]
for each $i$. Furthermore, by the similar arguments:
\[
\textstyle
\sum\limits_{s(j)=i} \delta_{0,j} \le \sum\limits_{\Ss^0_{i, \eta^*}(\lfloor\Gamma_i\rfloor)} 1 + (\Gamma_{i} {-} \lfloor\Gamma_{i}  \rfloor) = \Gamma_i
\]

Finally, by Proposition \ref{prop:prim_optimal} (\ref{eqn:prim_objectives}) holds for the optimal solution obtained by the cutting planes algorithm \ref{alg:cutting_planes_LP} and the optimal solution of (\ref{eqn:RatesLP-RC2}) and hence by the strong duality we have:
\[
\textstyle
 a\Tt p'  {+} [b\Tt\; 0] \dq' {=} \sum\limits_{k=1}^K \sum\limits_{\ell {\in} L(k)}   a_k p_{k, \ell}  {+} \sum\limits_{i=1}^{I} b_{i} \dq_i {=} a\Tt p''  {+} [b\Tt\; 0] \dq''  
\]
where $p'', \dq''$ is an optimal solution of (\ref{eqn:RatesLP-RC2*}). Thus $p', \dq', \dot{\lambda}$ is also the optimal solution of (\ref{eqn:RatesLP-RC2*}).
\end{proof}

\begin{proof}\emph{[Theorem \ref{thm:sclp_alg}]}
One can see that (\ref{eqn:RatesLP-RC2}) is a Rates-LP of (\ref{eqn:SCLP-RC2}) and (\ref{eqn:RatesLP-RC2*}) is a Rates-LP of the symmetric dual of (\ref{eqn:SCLP-RC2}). Thus, $x(t), \eta(t), \beta(t), \gamma(t), v(t), r(t)$ could be obtained from the bases of (\ref{eqn:RatesLP-RC2}) and solution of (\ref{eqn.boundary}). The SCLP-simplex algorithm ensure that $x(t)$ is non-negative for all $t$, while  $\eta(t), \beta(t), \gamma(t), v(t), r(t)$ are non-negative for all $t$ by construction. Similarly, $p(t), q(t), \delta(t), y(t), \omega(t)$ could be obtained from the bases of (\ref{eqn:RatesLP-RC2*}) and solution of (\ref{eqn.boundary}),  where non-negativity of $q(t)$ for all $t$ enforced by the SCLP-simplex algorithm, while $p(t), \delta(t), y(t), \omega(t)$ are non-negative by construction. Thus we have a pair of the feasible solutions of (\ref{eqn:SCLP-RC2}) and its symmetric dual. Furthermore, one can check that these solutions are complementary slack and hence are optimal. Therefore, algorithm \ref{alg:sclp_simplex} with modifications \ref{def:modification} provide a robust optimal solution for the uncertain SCLP (\ref{eqn:ASCLP}) with one-sided budgeted uncertainty set.
\end{proof}

\end{document}